\documentclass[10pt,a4paper,final]{amsart}
\usepackage[margin=2.8cm, marginparwidth=2.3cm]{geometry}
\usepackage[english]{babel}
\usepackage[utf8]{inputenc}
\usepackage[OT2,T1]{fontenc}
\usepackage{amssymb}
\usepackage{pbox}
\usepackage{commath}
\usepackage[colorlinks=true,citecolor=blue,linkcolor=black]{hyperref}

\usepackage[obeyDraft,color=green!40]{todonotes}
\usepackage[notref,notcite,color]{showkeys}
\definecolor{labelkey}{rgb}{1.0, 0.22, 0.0}

\newtheorem{theorem}{Theorem}
\newtheorem{lemma}[theorem]{Lemma}
\newtheorem{proposition}[theorem]{Proposition}
\newtheorem{corollary}[theorem]{Corollary}
\newtheorem{conjecture}{Conjecture}

\numberwithin{equation}{section}

\DeclareMathOperator{\rank}{rank}
\DeclareMathOperator{\corank}{corank}
\DeclareMathOperator{\tors}{tors}
\DeclareMathOperator{\Sel}{Sel}

\newcommand{\rankan}{\rank_{\mathrm{an}}}
\newcommand{\ee}{\mathrm{e}}

\DeclareSymbolFont{cyrletters}{OT2}{wncyr}{m}{n}
\DeclareMathSymbol{\Sha}{\mathalpha}{cyrletters}{"58}

\AtBeginDocument{\def\MR#1{}} 

\title[Average rank of quadratic twists]{Average rank of quadratic twists with a rational point of almost minimal height}
\author{Joachim Petit}
\address{Department of Mathematics and Computer Science \\ University of Basel \\ Spiegelgasse 1 \\ 4051 Basel \\ Switzerland}
\email{joachim.petit@unibas.ch}

\begin{document}

\begin{abstract}
	Given a family of quadratic twists of a fixed elliptic curve defined over $\mathbb{Q}$, we investigate the average rank in the subfamily of twists having a nontorsion rational point of almost minimal height.
	We show in particular that the average analytic rank is greater than one.
\end{abstract}

\maketitle
\tableofcontents
\thispagestyle{empty}

\section{Introduction} \label{sec:intro}

Fix $A,B \in \mathbb{Z}$ satisfying $4A^3 + 27B^2 \neq 0$.
Let $E$ denote the elliptic curve defined over the rationals by the equation
\[ E : y^2 = x^3 + Ax + B, \]
and for every square-free integer $d$, let $E_d$ be the quadratic twist of $E$ defined over the rationals by the equation
\[ E_d : dy^2 = x^ + Ax + B. \]

Write $\rankan(E_d)$ for the analytic rank of $E_d$, that is the order of vanishing of the Hasse--Weil $L$-function $L(E_d,s)$ at $s = 1$, and $\rank(E_d(\mathbb{Q}))$ for the algebraic rank of $E_d$, which is the rank of the finitely generated abelian group $E_d(\mathbb{Q})$.
Recall that the Parity Conjecture asserts that
\[ \rankan(E_d) \equiv \rank(E_d(\mathbb{Q})) \bmod 2, \]
while the Birch and Swinnerton-Dyer Conjecture states that the two ranks are actually equal.
For $X \geq 1$, denote the set of square-free integers up to $X$ by
\[ \mathcal{S}(X) = \{ 1 \leq d \leq X : \mu(d) \neq 0 \}. \]
It has been conjectured by Goldfeld~\cite{MR564926} that the average analytic rank of quadratic twists, when ordered by the size of $d$, is exactly one half.

\begin{conjecture}[Goldfeld] \label{conj:Goldfeld}
	Let $\nu \in \{0,1\}$. One has
	\[ \#\{ d \in \mathcal{S}(X) : \rankan(E_d) = \nu \} \sim \frac12 \#\mathcal{S}(X), \]
	as $X \to \infty$.
\end{conjecture}

Note that the same result is conjectured to hold with the algebraic rank replacing the analytic rank as a consequence of the Birch and Swinnerton-Dyer Conjecture.

Conjecture~\ref{conj:Goldfeld} has been studied extensively for curves $E$ having full rational $2$-torsion, starting with the work of Heath-Brown~\cite{MR1193603}, \cite{MR1292115} on the congruent number curve $E : y^2 = x^3 - x$.
These results were later extended by Kane~\cite{MR3101079} to any curve $E$ with full rational 2-torsion having no rational cyclic subgroup of order four.
For such curves, Smith~\cite{2017arXiv170202325S} has recently shown that
\[ \#\{ 1 \leq d \leq X : \rank(E_d(\mathbb{Q})) \geq 2 \} = o(X), \]
achieving significant progress towards Conjecture~\ref{conj:Goldfeld} for the algebraic rank.
Building upon this result, it was proven by Kriz~\cite{2020arXiv200204767K} that Conjecture~\ref{conj:Goldfeld} holds under the additional assumption that $E$ has complex multiplication by $\mathbb{Q}(i)$ or $\mathbb{Q}(\sqrt2)$.

We also mention that the prediction of Conjecture~\ref{conj:Goldfeld} has since been extended, following the heuristics of Katz and Sarnak~\cite{MR1659828}, to state that amongst all elliptic curves defined over $\mathbb{Q}$, the density of rank zero curves and the density of rank one curves are both one half.
This broader version of the conjecture is supported by the recent work of Bhargava and Shankar~\cite{MR3272925}, \cite{MR3275847}, \cite{2013arXiv1312.7333B}, \cite{2013arXiv1312.7859B}.

Le~Boudec~\cite{MR3788863} has recently investigated the analogue of Conjecture~\ref{conj:Goldfeld} when the twists are ordered using a quantity coming from the geometry of their Mordell--Weil lattice.
Recall that each twist $E_d$ comes equipped with a canonical height function $\hat{h}_{E_d} : E_d(\mathbb{Q}) \to \mathbb{R}_{\geq0}$ (see Section~\ref{sec:thmain} for its definition). Following the work of Le~Boudec~\cite{MR3455753}, we consider the quantity $\eta_d(A,B)$ defined through
\begin{equation} \label{eq:def:eta_d}
	\log\eta_d(A,B) = \min\{ \hat{h}_{E_d}(P) : P \in E_d(\mathbb{Q}) \setminus E_d(\mathbb{Q})_{\tors} \},
\end{equation}
if $\rank(E_d(\mathbb{Q})) \geq 1$, and by $\eta_d(A,B) = \infty$ if $\rank(E_d(\mathbb{Q})) = 0$.
Le~Boudec~\cite{MR3788863} showed that when the twists are ordered by the size of $\eta_d(A,B)$, then the average analytic rank is larger than one, and under the Parity Conjecture the same holds for the algebraic rank.

Remark that if $\eta_d(A,B) < \infty$, then $\rank(E_d(\mathbb{Q})) \geq 1$ which, by the work of Kolyvagin~\cite{MR954295} and Breuil, Conrad, Diamond and Taylor~\cite{MR1839918}, implies that $\rankan(E_d) \geq 1$.
Thus, given the prediction of Conjecture~\ref{conj:Goldfeld}, one might have expected the average rank to be one.

The goal of the present work is to provide a further example of a subfamily of the family of all quadratic twists for which the average analytic rank is also greater than one.
The quantity $\eta_d(A,B)$ defined in \eqref{eq:def:eta_d} satisfies the sharp lower bound (see for instance~\cite[Section~2.2]{MR3455753})
\[ \eta_d(A,B) \gg d^{1/8}, \]
and we are interested in the quadratic twists for which this bound is almost attained.
For $\alpha > 0$, we define the set
\[ \mathcal{D}_{A,B}(\alpha;X) = \{ d \in \mathcal{S}(X) : \eta_d(A,B) \leq d^{1/8+\alpha} \}. \]
We note that the author~\cite[Theorem~1]{2020arXiv200402500P} has recently established an asymptotic formula for the cardinality of this set when $\alpha \in (0,1/120)$. Our main result investigates the average analytic rank of the quadratic twists of $E$ as $d$ runs over $\mathcal{D}_{A,B}(\alpha;X)$.

\begin{theorem} \label{thm:main}
	Let $\alpha \in (0,1/120)$ and assume that the polynomial $x^3+Ax+B$ is irreducible over $\mathbb{Z}$. One has
	\[ \liminf_{X\to\infty} \frac{1}{\#\mathcal{D}_{A,B}(\alpha;X)} \sum_{d\in\mathcal{D}_{A,B}(\alpha;X)} \rankan(E_d) > 1. \]
\end{theorem}

We have decided to restrict ourselves to the case of an irreducible polynomial in order to keep the proof concise, but we expect our method to be robust enough to handle the general case.

For a prime number $p$, let $\Sel_{p^\infty}(E_d)$ be the $p^\infty$-Selmer group of $E_d$. There is a short exact sequence
\[ 0 \longrightarrow E_d(\mathbb{Q}) \otimes \mathbb{Q}_p / \mathbb{Z}_p \longrightarrow \Sel_{p^\infty}(E_d) \longrightarrow \Sha(E_d)[p^\infty] \longrightarrow 0, \]
where $\Sha(E_d)$ is the Tate--Shafarevich group of $E_d$ and $\Sha(E_d)[p^\infty]$ is its $p$-primary part.
Let $\corank(\Sel_{p^\infty}(E_d))$ be the number of copies of $\mathbb{Q}_p / \mathbb{Z}_p$ in $\Sel_{p^\infty}(E_d)$. The above short exact sequence implies the inequality
\begin{equation} \label{eq:ineqcorank}
	\corank(\Sel_{p^\infty}(E_d)) \geq \rank(E_d(\mathbb{Q})),
\end{equation}
while the Tate--Shafarevich Conjecture implies that these quantities are actually equal.
By \eqref{eq:ineqcorank}, we have $\corank(\Sel_{p^\infty} (E_d)) \geq 1$ for all $d \in \mathcal{D}_{A,B}(\alpha;X)$. Moreover, the $p$-Parity Theorem of Dokchitser and Dokchitser \cite[Theorem~1.4]{MR2680426} states that $\corank(\Sel_{p^\infty} (E_d))$ and $\rankan(E_d)$ share the same parity.

During the proof of Theorem~\ref{thm:main}, we actually show that there exists a positive proportion of twists with analytic rank even and at least two, which implies the following corollary.

\begin{corollary}
	Under the assumptions of Theorem~\ref{thm:main}, one has
	\[ \liminf_{X\to\infty} \frac{1}{\#\mathcal{D}_{A,B}(\alpha;X)} \sum_{d\in\mathcal{D}_{A,B}(\alpha;X)} \corank(\Sel_{p^\infty} (E_d)) > 1. \]
\end{corollary}

In addition, we also clearly obtain the following conditional result about the average algebraic rank.

\begin{corollary}
	Assume the Parity Conjecture. Under the assumptions of Theorem~\ref{thm:main}, one has
	\[ \liminf_{X\to\infty} \frac{1}{\#\mathcal{D}_{A,B}(\alpha;X)} \sum_{d\in\mathcal{D}_{A,B}(\alpha;X)} \rank(E_d(\mathbb{Q})) > 1. \]
\end{corollary}

\subsection*{Organisation of the paper}
The proof of Theorem~\ref{thm:main} follows the strategy of Le~Boudec~\cite{MR3788863} and is the subject of Section~\ref{sec:thmain}, where it is realised assuming the validity of Propositions~\ref{prop:RQ2} and~\ref{prop:SQnu}.
The purpose of Section~\ref{sec:SME} is to establish Proposition~\ref{prop:RQ2}, using a result of Salberger~\cite{MR2369057} based on the determinant method and which improves on the celebrated work of Heath-Brown~\cite{MR1906595}.
In Section~\ref{sec:SFS}, we establish Proposition~5 by showing that a certain polynomial in three variables takes the expected amount of square-free values as its variables run over some region.

\subsection*{Acknowledgements}
The research of the author is integrally funded by the Swiss National Science Foundation through the SNSF Professorship number 170565 awarded to Pierre Le~Boudec for the project \emph{Height of rational points on algebraic varieties}.

The author would like to thank the referee for their extremely careful work and very pertinent suggestions.

\section{The proof of Theorem~\ref{thm:main}} \label{sec:thmain}

Throughout this paper, we view $A$, $B$ and $\alpha$ as fixed and every implicit constant may depend on all of them. 
Any other dependence is specified.

We begin by recalling the definitions of various height functions. Let $h : \mathbb{P}^1(\mathbb{Q}) \to \mathbb{R}_{\geq0}$ denote the logarithmic Weil height and let $h_x : \mathbb{P}^2(\mathbb{Q}) \to \mathbb{R}_{\geq0}$ be defined by
\[ h_x(x:y:z) = h(x:z) \]
if $(x:y:z) \neq (0:1:0)$, and $h_x(0:1:0) = 0$.
The canonical height $\hat{h}_{E_d}$ on $E_d$ is defined for $P \in E_d(\mathbb{Q})$ by
\[ \hat{h}_{E_d}(P) = \frac12 \lim_{n\to\infty} \frac{1}{4^n} h_x(2^nP). \]
One can show (see for instance~\cite[Lemma~3]{MR3455753}) the existence of two constants $c_1$ and $c_2$ depending only on $A$ and $B$ such that for every $P \in E_d(\mathbb{Q})$, one has
\begin{equation} \label{eq:heightineq}
	c_1 < \hat{h}_{E_d}(P) - \frac12 h_x(P) < c_2.
\end{equation}
Write
\begin{equation} \label{eq:def:F}
	F(x,z) = x^3 + Axz^2 + Bz^3,
\end{equation}
and recall that, under the assumptions of Theorem~\ref{thm:main}, this polynomial is irreducible over the integers.
We introduce the polynomial
\begin{equation} \label{eq:def:Q}
	Q(u,v,w) = v F(u,vw^2).
\end{equation}
For $(u,v,w) \in \mathbb{Z}_{\geq1}^3$ satisfying $Q(u,v,w) \geq 1$ and $\mu(Q(u,v,w)) \neq 0$, the point $P$ with coordinates $(uvw : 1 : v^2w^3)$ lies in $E_{Q(u,v,w)}(\mathbb{Q})$. The square-free assumption on $Q(u,v,w)$ ensures that $\gcd(u,vw) = 1$, which implies $h_x(P) = \log\max\{u,vw^2\}$. Thus, if $\max\{u,vw^2\} \geq \ee^{-2c_1}$, the first inequality in \eqref{eq:heightineq} implies $\hat{h}_{E_d}(P) \neq 0$, so $P$ is not a torsion point. In this case, we have
\[ \log \eta_d(A,B) \leq \hat{h}_{E_d}(P) < \frac12 h_x(P) + c_2, \]
which leads to
\[ \eta_d(A,B) < \ee^{c_2} \max\{u,vw^2\}^{1/2}. \]
We introduce the region
\begin{equation} \label{eq:def:RQ}
	\mathcal{R}_Q(\alpha) = \{ (u,v,w) \in \mathbb{R}_{\geq1}^3 : \ee^{-2c_1} \leq \max\{u, vw^2\} \leq \ee^{-2c_2} Q(u,v,w)^{1/4+2\alpha} \},
\end{equation}
and define the number of representations
\begin{equation} \label{eq:def:rQ}
	r_Q(\alpha;d) = \#\{ (u,v,w) \in \mathbb{Z}^3 \cap \mathcal{R}_Q(\alpha) : Q(u,v,w) = d \}.
\end{equation}
By the above discussion, if $d$ is square-free, the condition $r_Q(\alpha;d) \geq 1$ ensures $\eta_d(A,B) \leq d^{1/8+\alpha}$. In other words, the set
\begin{equation} \label{eq:def:Qfamily}
	\mathcal{Q}(\alpha;X) = \{ d \in \mathcal{S}(X) : r_Q(\alpha;d) \geq 1 \}
\end{equation}
is a subset of $\mathcal{D}_{A,B}(\alpha;X)$.

Let $\omega(E_d)$ denote the sign of the functional equation of $L(E_d,s)$. As mentioned in Section~\ref{sec:intro}, the analytic rank of $E_d$ is at least one for $d \in \mathcal{D}_{A,B}(\alpha;X)$, so that $\rankan(E_d)$ is at least 2 and even if and only if $\omega(E_d) = 1$.
For $\nu \in \{-1,1\}$, we define
\[ \Omega_\nu(\alpha;X) = \frac{\#\{ d \in \mathcal{D}_{A,B}(\alpha;X) : \omega(E_d) = \nu \}}{\#\mathcal{D}_{A,B}(\alpha;X)}. \]
To prove Theorem~\ref{thm:main}, we establish that for $\alpha \in (0,1/120)$ and $\nu \in \{-1,1\}$, one has
\begin{equation} \label{eq:Omeganu>0}
	\liminf_{X\to\infty} \Omega_\nu(\alpha;X) > 0.
\end{equation}

Letting $d$ range over the set $\mathcal{Q}(\alpha;X)$ defined in \eqref{eq:def:Qfamily} instead of $\mathcal{D}_{A,B}(\alpha;X)$ provides the lower bound
\begin{equation} \label{eq:OmeganuQ}
	\Omega_\nu(\alpha;X) \geq \frac{\#\{ d \in \mathcal{Q}(\alpha;X) : \omega(E_d) = \nu \}}{\#\mathcal{D}_{A,B}(\alpha;X)}.
\end{equation}

We introduce the second moment
\[ R_Q^{(2)}(\alpha;X) = \sum_{d \leq X} r_Q(\alpha;d)^2, \]
as well as the modified first moment
\[ S_{Q,\nu}(\alpha;X) = \sum_{\substack{d \in \mathcal{S}(X) \\ \omega(E_d) = \nu}} r_Q(\alpha;d). \]
The Cauchy--Schwarz inequality, combined with the lower bound \eqref{eq:OmeganuQ}, yields
\begin{equation} \label{eq:OmeganuLB}
	\Omega_\nu(\alpha;X) \geq \frac{S_{Q,\nu}(\alpha;X)^2}{\#\mathcal{D}_{A,B}(\alpha;X) R_Q^{(2)}(\alpha;X)}.
\end{equation}

The cardinality of the set $\mathcal{D}_{A,B}(\alpha;X)$ is the subject of a recent article by the author, where it is shown~\cite[Theorem~1]{2020arXiv200402500P} that for $\alpha < 1/120$, there exists a positive constant $c(\alpha)$ such that one has
\begin{equation} \label{eq:ART01}
	\#\mathcal{D}_{A,B}(\alpha;X) \sim c(\alpha) X^{1/2} \log X.
\end{equation}

In Section~\ref{sec:SME}, we establish the following upper bound for the second moment.
\begin{proposition} \label{prop:RQ2}
	Let $\alpha \in (0,1/120)$. One has
	\[ R_Q^{(2)}(\alpha;X) \ll X^{1/2} \log X. \]
\end{proposition}

In Section~\ref{sec:SFS}, we prove the following lower bound for the modified first moment.
\begin{proposition} \label{prop:SQnu}
	Let $\alpha \in (0,1/56)$ and $\nu \in \{-1,1\}$. One has
	\[ S_{Q,\nu}(\alpha;X) \gg X^{1/2} \log X. \]
\end{proposition}

Putting together the inequality \eqref{eq:OmeganuLB}, the estimate \eqref{eq:ART01} and Propositions~\ref{prop:RQ2} and~\ref{prop:SQnu}, we obtain \eqref{eq:Omeganu>0}, which completes the proof of Theorem~\ref{thm:main}.

\section{The second moment estimate} \label{sec:SME}

In this section, we establish Proposition~\ref{prop:RQ2}.
We begin with an easy preliminary lemma concerning the range of the triples in $\mathcal{R}_Q(\alpha)$.

\begin{lemma} \label{lem:wX4alpha}
	Let $\alpha \in (0,1/24)$. For $(u,v,w) \in \mathbb{Z}^3 \cap \mathcal{R}_Q(\alpha)$ satisfying $Q(u,v,w) \leq X$, one has $w \ll X^{4\alpha}$.
\end{lemma}

\begin{proof}
	From the definition of $\mathcal{R}_Q(\alpha)$ in~\eqref{eq:def:RQ}, one has
	\[ \max\{u,vw^2\} \ll Q(u,v,w)^{1/4+2\alpha} \ll (v\max\{u,vw^2\}^3)^{1/4+2\alpha}, \]
	so that
	$\max\{u,vw^2\}^{1-24\alpha} \ll v^{1+8\alpha}$. This implies $w \ll v^{16\alpha/(1-24\alpha)}$, and it suffices to note that one has $vw^2 \ll X^{1/4+2\alpha}$ to complete the proof.
\end{proof}

We now turn to the computation of $R_Q^{(2)}(\alpha;X)$. Consider the diagonal term
\[ m_Q(\alpha;d) = \#\left\{ ((u_1,v_1,w_1), (u_2,v_2,w_2)) \in \mathbb{Z}^3 \times \mathbb{Z}^3 : \pbox{\textwidth}{
	$(u_1,v_1,w_1), (u_2,v_2,w_2) \in \mathcal{R}_Q(\alpha)$ \\
	$Q(u_1,v_1,w_1) = Q(u_2,v_2,w_2) = d$ \\
	$(u_1:v_1w_1^2) = (u_2:v_2w_2^2)$ } \right\}, \]
where the extra condition is an equality in $\mathbb{P}^1(\mathbb{Q})$, and let
\begin{equation} \label{eq:aQ=rQ2-mQ}
	a_Q(\alpha;d) = r_Q(\alpha;d)^2 - m_Q(\alpha;d).
\end{equation}
Let
\begin{align} \label{eq:def:MQ+AQ}
	M_Q(\alpha;X) = \sum_{d \leq X} m_Q(\alpha;d), && A_Q(\alpha;X) = \sum_{d \leq X} a_Q(\alpha;d).
\end{align}
We now provide upper bounds for these two quantities, in the form of two lemmas.

\begin{lemma} \label{lem:UB:MQ}
	Let $\alpha \in (0,1/24)$. One has
	\[ M_Q(\alpha;X) \ll X^{1/2} \log X. \]
\end{lemma}

\begin{proof}
	For a fixed value of $d \geq 1$, let $((u_1,v_1,w_1),(u_2,v_2,w_2))$ be a pair counted in $m_Q(\alpha;d)$.
	Write $u = \gcd(u_1,u_2)$, $u_1 = ux$ and $u_2 = uy$, where $x$ and $y$ are coprime positive integers.
	The projective condition defining $m_Q(\alpha;d)$ dehomogenises to
	\begin{align*}
		x u_2 = y u_1, && x v_2w_2^2 = y v_1w_1^2.
	\end{align*}
	From these expressions, we obtain
	\[ Q(u_2,v_2,w_2) = \frac{y^3 v_2}{x^3 v_1} Q(u_1,v_1,w_1), \]
	and since $Q(u_1,v_1,w_1) = d \geq 1$, the equation $Q(u_1,v_1,w_1) = Q(u_2,v_2,w_2)$ in the definition of $m_Q(\alpha;d)$ becomes $x^3 v_1 = y^3 v_2$, or equivalently $x^2 w_2 = y^2 w_1$.
	From the coprimality of $x$ and $y$, it follows that $x^2$ divides $w_1$ and that $x^3$ divides $v_2$.
	Writing $w_1 = w x^2$ and $v_2 = v x^3$ for some positive integers $v$ and $w$, we obtain $w_2 = w y^2$ and $v_1 = v y^3$ and therefore, we have
	\[ ((u_1,v_1,w_1),(u_2,v_2,w_2)) = ((ux, vy^3, wx^2), (uy, vx^3, wy^2)), \]
	for some 5-tuple $(u, v, w, x, y)$ of positive integers. This brings about the new expression
	\[ m_Q(\alpha;d) \leq \#\left\{ (u,v,w,x,y) \in \mathbb{Z}_{\geq1}^5 : \pbox{\textwidth}{
		$(ux, vy^3, wx^2), (uy, vx^3, wy^2) \in \mathcal{R}_Q(\alpha)$ \\
		$Q(ux, vy^3, wx^2) = d$ } \right\}, \]
	and Lemma~\ref{lem:wX4alpha} implies $wx^2, wy^2 \ll X^{4\alpha}$.
	Summing over $d$, we reach the upper bound
	\begin{equation} \label{eq:UB:MQ}
		M_Q(\alpha;X) \ll \mathop{\sum\sum\sum}_{wx^2, wy^2 \ll X^{4\alpha}}  \#\left\{ (u,v) \in \mathbb{Z}_{\geq1}^2 : \pbox{\textwidth}{
			$1 \leq Q(ux, vy^3, wx^2) \leq X$ \\
			$ux, uy, vw^2x^3y^4, vw^2x^4y^3 \ll X^{1/4+2\alpha}$ } \right\}.
	\end{equation}

	An elementary application of the classical lattice-point counting method of Davenport~\cite{MR43821} shows
	\[ \#\left\{ (u,v) \in \mathbb{Z}_{\geq1}^2 : \pbox{\textwidth}{
		$1 \leq Q(ux, vy^3, wx^2) \leq X$ \\
		$ux, uy, vx, vy \ll X^{1/4+2\alpha}$ } \right\}
		\ll \iint_{0 < Q(ux,vy^3,wx^2) \leq X} \dif u \dif v + \frac{X^{1/4+2\alpha}}{\max\{x,y\}}. \]
	The change of variables $(u,v) \mapsto (uvX^{1/4}w^{1/2} (xy)^3, vX^{1/4}w^{-3/2})$ yields
	\[ \iint_{0 < Q(ux,vy^3,wx^2) \leq X} \dif u \dif v = \frac{X^{1/2} (xy)^3}{w} \iint_{0 < (xy)^{12} v^4 F(u,1) \leq 1} v \dif u \dif v. \]
	The integral over $v$ has size $2^{-1}F(u,1)^{-1/2}(xy)^{-6}$ and, since $F(u,1)$ has three distinct roots, integrating over $u$ gives a constant times $(xy)^{-6}$.
	Summing over $x$, $y$ and $w$ in \eqref{eq:UB:MQ}, we obtain 
	\begin{align*}
		M_Q(\alpha;X) & \ll \mathop{\sum\sum\sum}_{wx^2, wy^2 \ll X^{4\alpha}} \left\{ \frac{X^{1/2}}{w(xy)^3} + \frac{X^{1/4+2\alpha}}{\max\{x,y\}} \right\} \\
			& \ll X^{1/2} \log X + X^{1/4+6\alpha},
	\end{align*}
	which concludes the proof since $\alpha < 1/24$.
\end{proof}

\begin{lemma} \label{lem:UB:AQ}
	Let $\alpha \in (0,1/120)$. One has
	\[ A_Q(\alpha;X) \ll X^{1/2}. \]
\end{lemma}

\begin{proof}
	Using Lemma~\ref{lem:wX4alpha} and carrying out the summation over $d$, we find in particular that
	\[ A_Q(\alpha;X) \ll \mathop{\sum\sum}_{w_1,w_2 \ll X^{4\alpha}} \#\left\{ (u_1,u_2,v_1,v_2) \in \mathbb{Z}^4 : \pbox{\linewidth}{
		$1 \leq u_1, u_2, v_1, v_2 \ll X^{1/4+2\alpha}$ \\
		$Q(u_1,v_1,w_1) = Q(u_2,v_2,w_2) $ \\
		$(u_1:v_1w_1^2) \neq (u_2:v_2w_2^2)$ } \right\}. \]
	For fixed $w_1$ and $w_2$, write $\mathbf{w} = (w_1,w_2)$ and let $V_{\mathbf{w}}$ be the projective surface in $\mathbb{P}^3$ defined by the equation
	\[ Q(u_1,v_1,w_1) = Q(u_2,v_2,w_2). \]
	Let $H$ denote the usual exponential Weil height on $\mathbb{P}^3(\mathbb{Q})$ and set
	\[ a_{Q,\mathbf{w}}(k;\alpha;X) = \#\left\{ (u_3 : u_4 : v_3 : v_4) \in V_{\mathbf{w}}(\mathbb{Q}) : \pbox{\linewidth}{
		$H(u_3 : u_4 : v_3 : v_4) \ll X^{1/4+2\alpha}/k$ \\
		$(u_3:v_3w_1^2) \neq (u_4:v_4w_2^2)$ } \right\}. \]
	Summing over $k = \gcd(u_1,u_2,v_1,v_2)$, we find
	\[ A_Q(\alpha;X) \ll \mathop{\sum\sum}_{w_1,w_2 \ll X^{4\alpha}} \sum_{k \ll X^{1/4+2\alpha}} a_{Q,\mathbf{w}}(k;\alpha;X). \]

	Let $L_{\mathbf{w}}$ denote the Zariski closed subset of $V_{\mathbf{w}}$ defined as the union of the lines contained in $V_{\mathbf{w}}$.
	Since the polynomial $F(x,z)$ is assumed to be irreducible, meaning that the curve $E$ and all its twists $E_d$ have trivial 2-torsion, a recent result by the author~\cite[Corollary~4.6]{2020arXiv200402500P} states that if a point $(u_3:u_4:v_3:v_4) \in V_{\mathbf{w}}(\mathbb{Q})$ lies in $L_{\mathbf{w}}(\mathbb{Q})$, then $(u_3:v_3w_1^2) = (u_4:v_4w_2^2)$ in $\mathbb{P}^1(\mathbb{Q})$. Such points are excluded here and a result of Salberger~\cite[Theorem~0.1]{MR2369057} therefore shows that for any $\epsilon > 0$, we have
	\[ a_{Q,\mathbf{w}}(k;\alpha;X) \ll_\epsilon \left( \frac{X^{1/4+2\alpha}}{k} \right)^{13/8+\epsilon}. \]
	It follows that
	\[ A_Q(\alpha;X) \ll_\epsilon X^{13/32+45\alpha/4+\epsilon}, \]
	and the assumption $\alpha < 1/120$ concludes the proof.
\end{proof}

We finally note that combining the equality~\eqref{eq:aQ=rQ2-mQ} with the definitions in~\eqref{eq:def:MQ+AQ} and with Lemmas~\ref{lem:UB:MQ} and~\ref{lem:UB:AQ}, we immediately complete the proof of Proposition~\ref{prop:RQ2}.

\section{The square-free sieve} \label{sec:SFS}

This section is dedicated to establishing Proposition~\ref{prop:SQnu}.
Recall that $N_E$ denotes the conductor of the curve $E$. 

For $d \geq 1$ and square-free, let $D$ be the discriminant of the quadratic field $\mathbb{Q}(\sqrt{d})$ and let $\chi_d$ be the associated quadratic character. When $\gcd(D,N_E) = 1$, the sign of the functional equation of $L(E_d,s)$ is determined by (see~\cite[Section~4]{MR2322343})
\[ \omega(E_d) = \chi_d(-N_E) \omega(E). \]
Using this equality and of the fact that $-N_E$ is not a square, which implies that $\chi_d(-N_E)$ attains both signs as $d$ varies, it follows from quadratic reciprocity that there exists $d_0 \in (\mathbb{Z}/4N_E\mathbb{Z})^\times$ such that
\[ S_{Q,\nu}(\alpha;X) \geq \sum_{\substack{d \in \mathcal{S}(X) \\ d \equiv d_0 \bmod 4N_E}} r_Q(\alpha;d). \]

Replacing $r_Q(\alpha;d)$ by its definition in~\eqref{eq:def:rQ}, this inequality becomes
\[ S_{Q,\nu}(\alpha;X) \geq \#\left\{ (u,v,w) \in \mathbb{Z}^3 \cap \mathcal{R}_Q(\alpha) : \pbox{\textwidth}{
	$Q(u,v,w) \in \mathcal{S}(X)$ \\
	$Q(u,v,w) \equiv d_0 \bmod 4N_E$ } \right\}. \]
	
It is immediate from the definition of $Q(u,v,w)$ in~\eqref{eq:def:Q} that $Q(1,d_0,0) = d_0$, so defining
\begin{equation} \label{eq:def:NalphaX}
	N(\alpha;X) = \#\left\{ (u,v,w) \in \mathbb{Z}^3 \cap \mathcal{R}_Q(\alpha) : \pbox{\textwidth}{
		$Q(u,v,w) \in \mathcal{S}(X)$ \\
		$(u,v,w) \equiv (1,d_0,0) \bmod 4N_E$ } \right\},
\end{equation}
we obtain the inequality
\[ S_{Q,\nu}(\alpha;X) \geq N(\alpha;X). \]

The remainder of Section~\ref{sec:SFS} is dedicated to the proof of Propositon~\ref{prop:NalphaX} below, which implies Proposition~\ref{prop:SQnu}.

\begin{proposition} \label{prop:NalphaX}
	Let $\alpha \in (0,1/56)$. One has
	\[ N(\alpha;X) \gg X^{1/2} \log X. \]
\end{proposition}

We point out that a square-free sieve argument has already been employed by Gouvêa and Mazur~\cite{MR1080648}, and at the same time but independently by Greaves~\cite{MR1150469}, to show that integral binary forms, under certain assumptions, represent the expected number of square-free integers.
This is similar to the statement of Proposition~\ref{prop:NalphaX}, the difference being that in our case the variables range over the region defined in~\eqref{eq:def:RQ} instead of over a box as they do in both aforementioned articles.
For this reason, the sieving process must be carried out in full here.

\subsection{Preliminary results} \label{ssec:SFS:prelim}

This section presents several results required later on.
Recall the definition~\eqref{eq:def:F} of the binary cubic form $F(x,z)$ and let $\Delta_F = -(4A^3+27B^2)$ denote the discriminant of the polynomial $F(x,1)$.
Define the arithmetic function
\begin{equation} \label{eq:def:rho}
	\rho(n) = \#\{ u \bmod n : F(u,1) \equiv 0 \bmod n \}.
\end{equation}
This function is multiplicative and satisfies the upper bound~\cite[Corollary~2]{MR1119199}
\begin{equation} \label{eq:rhoSte91}
	\rho(p^k) \leq 2 p^{v_p(\Delta_F)/2} + 1,
\end{equation}
for any prime $p$ and any $k \geq 1$.
Moreover, it follows from Hensel's lemma that for $p \nmid \Delta_F$ and $k \geq 1$, one has
\begin{equation} \label{eq:rhopkrhop}
	\rho(p^k) = \rho(p).
\end{equation}
Together, \eqref{eq:rhoSte91} and \eqref{eq:rhopkrhop} imply $\rho(n^k) \ll \rho(n)$. By way of an estimate for the summatory function of $\rho(n)$ (e.g.~\cite[Equation~4]{MR1691404}), this demonstrates the upper bound
\begin{equation} \label{eq:sumrho}
	\sum_{n \leq X} \rho(n^k) \ll X.
\end{equation}

Finally, we give an elementary lemma about a sum of a certain arithmetic function involving $\varphi^*(n) = \varphi(n)/n$.

\begin{lemma} \label{lem:SFS:MT:v}
	Let $a, q, m \in \mathbb{Z}_{\geq1}$ with $\gcd(a,q) = 1$. One has
	\[ \sum_{\substack{n \leq X \\ \gcd(n,m) = 1 \\ n \equiv a \bmod q}} \mu(n)^2 \varphi^*(n)
		= \frac{C_0(q) C_1(m;q)}{q} X + O(X^{1/2}), \]
	with
	\begin{align*}
		C_0(q) = \prod_{p \nmid q} \left( 1 - \frac{2}{p^2} + \frac{1}{p^3} \right), &&
		C_1(m;q) = \prod_{\substack{p \mid m \\ p \nmid q}} \left(1 + \frac{1}{p} - \frac{1}{p^2} \right)^{-1}.
	\end{align*}
\end{lemma}

\begin{proof}
	Write
	\[ f(n) = \begin{cases}
		\mu(n)^2 \varphi^*(n), & \text{if } \gcd(n,m) = 1, \\
		0, & \text{otherwise}, \end{cases} \]
	and denote the quantity of interest by
	\[ A(X) = \sum_{\substack{n \leq X \\ n \equiv a \bmod q}} f(n). \]
	This sum can be expressed as
	\begin{align*}
		A(X) 	& = \sum_{\substack{n \leq X \\ n \equiv a \bmod q}} \sum_{d \mid n} (f \ast \mu)(d) \\
			& = \frac{X}{q} \sum_{\substack{d \geq 1 \\ \gcd(d,q) = 1}} \frac{(f\ast\mu)(d)}{d}
				+ O\left( \sum_{d \leq X} |(f\ast\mu)(d)| \right)
				+ O\left( \frac{X}{q} \sum_{d > X} \frac{|(f\ast\mu)(d)|}{d} \right).
	\end{align*}
	Remark that for $k \geq 1$, one has
	\[ (f\ast\mu)(p^k) = \begin{cases}
		0, 		& p \nmid m \text{ and } k \geq 3 \text{ or } p \mid m \text{ and } k \geq 2, \\
		-1/p, 		& p \nmid m \text{ and } k = 1, \\
		-(1-1/p), 	& p \nmid m \text{ and } k = 2, \\
		-1, 		& p \mid m \text{ and } k =1, \end{cases} \]
	so the main term is as claimed.
	Remark that the computation of $(f\ast\mu)(p^k)$ shows, by multiplicativity, that one has
	\[ | (f\ast\mu)(d) | \leq | (f\ast\mu)(d') |, \]
	for all $d \geq 1$, where
	\[ d' = \prod_{p \nmid m} p^{v_p(d)}. \]
	One can therefore assume $m = 1$ when estimating the error terms, which we now do.

	Let $d \geq 1$ cube-free and write $d = d_1d_2^2$ with $d_1$ and $d_2$ square-free and coprime. One has
	\begin{align*}
		|(f\ast\mu)(d_1)| = 1/d_1, && |(f\ast\mu)(d_2^2)| \leq 1,
	\end{align*}
	and so, the sums appearing in the error terms are bounded as
	\[ \sum_{d \leq X} |(f\ast\mu)(d)| \leq \mathop{\sum\sum}_{d_1d_2^2 \leq X} \frac{1}{d_1} \ll X^{1/2}, \]
	and
	\[ \sum_{d > X} \frac{|(f\ast\mu)(d)|}{d} \leq \mathop{\sum\sum}_{d_1d_2^2 > X} \frac{1}{d_1^2d_2^2} \ll X^{-1/2}. \]
	This concludes the proof.
\end{proof}

\subsection{Setup for the square-free sieve} \label{ssec:SFS:setup}

In this section, we establish a lower bound for the quantity $N(\alpha;X)$ defined in~\eqref{eq:def:NalphaX} by transforming the region in which the triples $(u,v,w)$ range.

Impose the additional condition $u \geq c_3vw^2$ for some $c_3 \geq 1$, so the maximum in the definition of $\mathcal{R}_Q(\alpha)$ in~\eqref{eq:def:RQ} is simply $u$, and let $c_4 = 1 + |A| + |B|$. For $c_3$ large enough,
one has
\[ \frac12 u^3 \leq F(u,vw^2) \leq c_4 u^3. \]
These bounds allow us to strengthen the two conditions $Q(u,v,w) \leq X$ and $u \leq \ee^{-2c_2} Q(u,v,w)^{1/4+2\alpha}$ to 
\begin{align*}
	c_4u^3v \leq X, && u \leq \ee^{-2c_2}  \left(\frac{u^3v}{2} \right)^{1/4+2\alpha}.
\end{align*}
Letting $c_5 = \ee^{-8c_2} 2^{-(1+8\alpha)}$, the last condition can be rewritten as $u^{1-24\alpha} \leq c_5 v^{1+8\alpha}$ and we therefore obtain the inequality
\[ N(\alpha;X) \geq \#\left\{ (u,v,w) \in \mathbb{Z}_{\geq1}^3 : \pbox{\textwidth}{
	$c_3vw^2 \leq u$, $c_4u^3v \leq X$ \\
	$u^{1-24\alpha} \leq c_5 v^{1+8\alpha}$ \\
	$\mu(Q(u,v,w)) \neq 0$ \\
	$(u,v,w) \equiv (1,d_0,0) \bmod 4N_E$ } \right\}. \]
Note that for any triple counted in this set, one has
\[ c_5 u v^{1+8\alpha} X^{8\alpha} \geq c_3 u^{1-24\alpha} vw^2 (c_4u^3v)^{8\alpha}, \]
and that setting $c_6 = c_3^{-1/2} c_4^{-4\alpha} c_5^{1/2}$, this inequality simplifies to
\[ w \leq c_6 X^{4\alpha}. \]

Let 
\begin{equation} \label{eq:def:Uw}
	U_w = c_3^{1/4} c_4^{-1/4} X^{1/4} w^{1/2}.
\end{equation}
We reduce the range of $u$ by imposing the additional restriction $u \leq U_w$, making the condition $c_4u^3v \leq X$ redundant and thus giving
\[ N(\alpha;X) \geq \sum_{\substack{w \leq c_6 X^{4\alpha} \\ w \equiv 0 \bmod 4N_E}} \#\left\{ (u,v) \in \mathbb{Z}_{\geq1}^2 : \pbox{\textwidth}{
	$c_3vw^2 \leq u \leq U_w$ \\
	$u^{1-24\alpha} \leq c_5 v^{1+8\alpha}$ \\
	$\mu(Q(u,v,w)) \neq 0$ \\
	$(u,v) \equiv (1,d_0) \bmod 4N_E$ } \right\}. \]

Next, we replace the summand in this estimate by the slightly larger quantity
\begin{equation} \label{eq:def:Nw}
	N_w(X) = \#\left\{ (u,v) \in \mathbb{Z}_{\geq1}^2 : \pbox{\textwidth}{
		$c_3vw^2 \leq u \leq U_w$ \\
		$\mu(Q(u,v,w)) \neq 0$ \\
		$(u,v) \equiv (1,d_0) \bmod 4N_E$ } \right\},
\end{equation}
which differs from said summand by an error term of size at most
\[ \#\left\{ (u,v) \in \mathbb{Z}_{\geq1}^2 : \pbox{\textwidth}{
	$u \leq U_w$ \\
	$u^{1-24\alpha} \gg v^{1+8\alpha}$ } \right\} 
	\ll (X^{1/2}w)^{(1-8\alpha)/(1+8\alpha)}. \]
We finally obtain
\begin{equation} \label{eq:N>sumNw}
	N(\alpha;X) \geq \sum_{\substack{w \leq c_6 X^{4\alpha} \\ w \equiv 0 \bmod 4N_E}} N_w(X) + O(X^{1/2}).
\end{equation}

\subsection{Applying inclusion-exclusion}

We now fix $w$ divisible by $4N_E$ and estimate the corresponding quantity $N_w(X)$ defined in~\eqref{eq:def:Nw}. 
From the definition of $Q(u,v,w)$ in \eqref{eq:def:Q}, it ensues that the condition $\mu(Q(u,v,w)) \neq 0$ appearing in \eqref{eq:def:Nw} can be reformulated as 
\[ \mu(v)\mu(F(u,vw^2)) \neq 0, \]
and we will apply a square-free sieve argument to the condition $\mu(F(u,vw^2)) \neq 0$.
Note that this condition implies $\gcd(u,vw) = 1$, a condition we will keep throughout the sieving process.

We remove the square-free condition on $F(u,vw^2)$ by summing over the integers $\ell \geq 1$ whose square divides $F(u,vw^2)$.
The inequalities $vw^2 \leq c_3vw^2 \leq u \leq U_w$ imply $F(u,vw^2) \leq c_4 U_w^3$, which restricts the range of $\ell$ to $\ell^2 \leq c_4 U_w^3$.
Let
\begin{equation} \label{eq:def:Nwl}
	N_{w,\ell}(X) = \#\left\{ (u,v) \in \mathbb{Z}_{\geq1}^2 : \pbox{\textwidth}{
		$c_3vw^2 \leq u \leq U_w$ \\
		$\mu(v) \neq 0$, $\gcd(u,vw)=1$ \\
		$F(u,vw^2) \equiv 0 \bmod \ell^2$ \\
		$(u,v) \equiv (1,d_0) \bmod 4N_E$ } \right\}, 
\end{equation}
and note that for every pair $(u,v)$ counted in this expression, one has $\gcd(\ell,w) = 1$.
An application of the inclusion-exclusion principle now gives
\begin{equation} \label{eq:Nw>sumNwl}
	N_w(X) = \sum_{\substack{\ell^2 \leq c_4 U_w^3 \\ \gcd(\ell,w)=1}} \mu(\ell) N_{w,\ell}(X).
\end{equation}

Note that all $v$ counted in~\eqref{eq:def:Nwl} are coprime to $\ell$ so summing over residue classes, that expression becomes
\[ N_{w,\ell}(X) = \mathop{\sum\sum}_{\substack{a,b \bmod \ell^2 \\ F(a,bw^2) \equiv 0 \bmod \ell^2 \\ \gcd(b,\ell) = 1}}
	 \#\left\{ (u,v) \in \mathbb{Z}_{\geq1}^2 : \pbox{\textwidth}{
	 	$c_3vw^2 \leq u \leq U_w$ \\
	 	$\mu(v) \neq 0$, $\gcd(u,vw) = 1$ \\
		$(u,v) \equiv (a,b) \bmod \ell^2$ \\
		$(u,v) \equiv (1,d_0) \bmod 4N_E$ } \right\}. \]
Shifting $a$ by a factor $bw^2$ simplifies the congruence condition and we arrive at
\begin{equation} \label{eq:Nwl-cc}
	N_{w,\ell}(X) = \sum_{\substack{a \bmod \ell^2 \\ F(a,1) \equiv 0 \bmod \ell^2}}
		 \#\left\{ (u,v) \in \mathbb{Z}_{\geq1}^2 : \pbox{\textwidth}{
		 	$c_3vw^2 \leq u \leq U_w$ \\
			$\mu(v) \neq 0$, $\gcd(u,vw) = 1$ \\
			$u \equiv avw^2 \bmod \ell^2$ \\
			$(u,v) \equiv (1,d_0) \bmod 4N_E$ } \right\}.
\end{equation}

\subsection{Reducing the range of $\ell$}

In this section, we show that the range of $\ell$ in~\eqref{eq:Nw>sumNwl} can be reduced at a negligible cost.
Recall the definition of $\rho(n)$ in~\eqref{eq:def:rho}. We begin with a simple upper bound on the size of $N_{w,\ell}(X)$.

\begin{lemma} \label{lem:UB:Nwl}
	One has
	\[ N_{w,\ell}(X) \ll \left(\frac{X^{1/2}}{w \ell^2} + 1\right) \rho(\ell^2). \]
\end{lemma}

\begin{proof}
	Abandoning some of the conditions in \eqref{eq:Nwl-cc} and extending the range of $u$ and $v$ yields
	\begin{align*}
		N_{w,\ell}(X) 
		& \ll \sum_{\substack{a \bmod \ell^2 \\ F(a,1) \equiv 0 \bmod \ell^2}}
		 \#\left\{ (u,v) \in \mathbb{Z}^2 : \pbox{\textwidth}{
		 	$1 \leq u, c_3vw^2 \leq U_w$ \\
			$\gcd(u,v) = 1$ \\
			$u \equiv avw^2 \bmod \ell^2$ } \right\} \\
		& \ll \sum_{\substack{a \leq \ell^2 \\ F(a,1) \equiv 0 \bmod \ell^2}}
		 \#\left\{ (u,v,k) \in \mathbb{Z}^3 : \pbox{\textwidth}{
			$|u| \leq U_w$ \\
			$|v| \ll U_w w^{-2}$ \\ 
			$|k| \ll U_w \ell^{-2}$ \\
			$\gcd(u,v,k) = 1$ \\
			$u - avw^2 - k \ell^2 = 0$ } \right\} 		
	\end{align*}
	A result of Heath-Brown~\cite[Lemma~3]{MR757475} shows that the last summand has size at most
	\[ \frac{U_w^2}{w^2\ell^2} + 1, \]
	which, combined with the definition of $U_w$ in \eqref{eq:def:Uw}, concludes the proof.
\end{proof}

An immediate consequence is the following result.

\begin{lemma} \label{lem:UB:Nw-large}
	Let $\theta > 0$. One has
 	\[ \sum_{X^\theta < \ell \ll U_w^{3/2}} \mu(\ell) N_{w,\ell}(X) \ll \frac{X^{1/2-\theta}}{w} + X^{3/8} w^{3/4}. \]
\end{lemma}

\begin{proof}
	Applying Lemma~\ref{lem:UB:Nwl}, we find
	\[ \sum_{X^\theta < \ell \ll U_w^{3/2}} \mu(\ell) N_{w,\ell}(X) \ll \frac{X^{1/2}}{w} \sum_{\ell > X^\theta} \frac{\rho(\ell^2)}{\ell^2} +  \sum_{\ell \leq U_w^{3/2}} \rho(\ell^2). \]
	Using \eqref{eq:sumrho} and, for the first sum, Abel summation, we obtain
	\begin{align*}
		\sum_{\ell > X^\theta} \frac{\rho(\ell^2)}{\ell^2} \ll \frac{1}{X^\theta}, && \sum_{\ell \ll U_w^{3/2}} \rho(\ell^2) \ll X^{3/8} w^{3/4},
	\end{align*}
	which completes the proof.
\end{proof}

Going back to~\eqref{eq:Nw>sumNwl}, Lemma~\ref{lem:UB:Nw-large} shows that for any $\theta > 0$, we have
\begin{equation} \label{eq:Nw>sumNwl-small}
	N_w(X) \geq \sum_{\substack{\ell < X^\theta \\ \gcd(\ell,w) = 1}} \mu(\ell) N_{w,\ell}(X) + O\left(\frac{X^{1/2-\theta}}{w}\right) + O(X^{3/8}w^{3/4}).
\end{equation}

\subsection{Estimating $N_{w,\ell}(X)$}

This section is dedicated to establishing a precise estimate for the quantity $N_{w,\ell}(X)$ defined in~\eqref{eq:def:Nwl} when $\gcd(\ell,w) = 1$. 
Removing the coprimality condition on $u$ in \eqref{eq:Nwl-cc}, we find
\[ N_{w,\ell}(X) = \sum_{\substack{a \bmod \ell^2 \\ F(a,1) \equiv 0 \bmod \ell^2}}
	\sum_{\substack{c_3vw^2 \leq U_w \\ \gcd(v,\ell) = 1 \\ v \equiv d_0 \bmod 4N_E}} \mu(v)^2
	\sum_{f \mid vw} \mu(f) \#\left\{ u_1 \in \mathbb{Z}_{\geq1} : \pbox{\textwidth}{
		$c_3vw^2 \leq fu_1 \leq U_w$ \\
		$fu_1 \equiv avw^2 \bmod \ell^2$ \\
		$fu_1 \equiv 1 \bmod 4N_E$ } \right\}. \]

One sees that $f$ can be inverted modulo $4N_E$, and also modulo $\ell^2$ as $f$ divides $vw$, which is coprime to $\ell$.
Recall that we assume that $4N_E$ divides $w$ so in particular, one has $\gcd(\ell, 4N_E) = 1$.
Together, these observations show that the two congruence conditions on $fu_1$ reduce to a single condition on $u_1$ modulo $4N_E\ell^2$ and we obtain the estimate
\[ \#\left\{ u_1 \in \mathbb{Z}_{\geq1} : \pbox{\textwidth}{
	$c_3vw^2 \leq fu_1 \leq U_w$ \\
	$fu_1 \equiv avw^2 \bmod \ell^2$ \\
	$fu_1 \equiv 1 \bmod 4N_E$ } \right\}
	= \frac{U_w - c_3 vw^2}{4N_E \ell^2 f} + O(1). \]
Recalling the definition of $\rho(n)$ in~\eqref{eq:def:rho}, we arrive at
\[ N_{w,\ell}(X) = \rho(\ell^2) \sum_{\substack{c_3vw^2 \leq U_w \\ \gcd(v,\ell) = 1 \\ v \equiv d_0 \bmod 4N_E}} \mu(v)^2
	\sum_{f \mid vw} \mu(f) \left\{ \frac{U_w - c_3 vw^2}{4N_E \ell^2 f} + O(1) \right\}. \]

We now define
\begin{equation} \label{eq:def:Vwl}
	V_{w,\ell}(X) = \sum_{\substack{c_3vw^2 \leq U_w \\ \gcd(v,\ell) = 1 \\ v \equiv d_0 \bmod 4N_E}} \mu(v)^2  \varphi^*(vw) ( U_w - c_3 vw^2 ),
\end{equation}
and obtain
\begin{equation} \label{eq:Nwl=Vwl}
	N_{w,\ell}(X) = \frac{\rho(\ell^2)}{4N_E \ell^2} V_{w,\ell}(X) + O_\epsilon\left( \frac{X^{1/4+\epsilon}}{w^{3/2}} \right),
\end{equation}
when $\gcd(\ell,w) = 1$, and zero otherwise. Here we have used the estimates $\rho(\ell^2) \ll_\epsilon \ell^\epsilon \ll X^\epsilon$ and $\#\{f \mid vw\} \ll_\epsilon (vw)^\epsilon \ll X^\epsilon$ to simplify the error term, and replaced $U_w$ by its definition~\eqref{eq:def:Uw}.

Define the two arithmetic functions
\begin{align} \label{eq:def:f1f2}
	f_1(n) = \sideset{}{'}\prod_{p \mid n} \left(1 + \frac{1}{p} - \frac{1}{p^2} \right)^{-1},
	&& f_2(n) = \varphi^*(w) \sideset{}{'}\prod_{p \mid n} \left( 1 - \frac{p}{p^2+2p-1} \right)^{-1},
\end{align}
where the prime indicates that the product is restricted to $p \nmid 4N_E$.
We compute $V_{w,\ell}(X)$ in the following lemma.

\begin{lemma} \label{lem:Vwl}
	Assume $\gcd(\ell, w) = 1$. There exists a constant $c_7 > 0$ such that one has
	\[ V_{w,\ell}(X) = \frac{c_7 f_2(w) f_1(\ell)}{w} X^{1/2} + O_\epsilon\left( \frac{X^{3/8}}{w^{1/4-\epsilon}} \right). \]
\end{lemma}

\begin{proof}
	Taking $g = \gcd(v,w)$ out, writing $v = gv_1$ and noting that $g$ and $\ell$ are necessarily coprime since we assume $\gcd(\ell, w) = 1$, the expression \eqref{eq:def:Vwl} becomes
	\[ V_{w,\ell}(X) = \sum_{g \mid w}  \sum_{\substack{c_3gv_1w^2 \leq U_w \\ \gcd(v_1,\ell) = 1 \\ gv_1 \equiv d_0 \bmod 4N_E}} \mu(gv_1)^2  \varphi^*(gv_1w) ( U_w - c_3 gv_1w^2 ). \]
	Note that one has $\gcd(v_1,g) = 1$ for every nonzero summand.
	Furthermore, recall from the beginning of Section~\ref{ssec:SFS:setup} that $d_0$ is invertible modulo $4N_E$, meaning that the sum over $v_1$ is zero unless $\gcd(g,4N_E)=1$. 
	Noting that one has $\varphi^*(gv_1w) = \varphi^\ast(v_1w) = \varphi^\ast(v_1) \varphi^*(w)$ as $g$ divides $w$ and $\gcd(v_1,w) = 1$ whenever $\mu(v_1) \neq 0$, we reach the new expression 
	\[ V_{w,\ell}(X) = \varphi^*(w) \sideset{}{'}\sum_{g \mid w} \mu(g)^2
		\sum_{\substack{c_3gv_1w^2 \leq U_w \\ \gcd(v_1,g\ell)=1 \\ v_1 \equiv d_0 \bar{g} \bmod 4N_E}} \mu(v_1)^2 \varphi^*(v_1) ( U_w - c_3 gv_1w^2 ). \]
	Here, the prime indicates that the sum is restricted to values of $g$ coprime to $4N_E$, and $\bar{g}$ denotes the inverse of $g$ modulo $4N_E$.
	Applying Lemma~\ref{lem:SFS:MT:v} and using Abel summation for the second summand, we find that the sum over $v_1$ is
	\[ \frac{C_0(4N_E) C_1(g\ell; 4N_E)}{8N_E c_3 g w^2} U_w^2 + O\left(\frac{U_w^{3/2}}{g^{1/2}w}\right), \]
	with $C_0(4N_E)$ and $C_1(g\ell; 4N_E)$ as in Lemma~\ref{lem:SFS:MT:v}.  Let
	\[ c_7 = \frac{C_0(4N_E)}{8 N_E (c_3c_4)^{1/2}} > 0, \]
	and remark that the arithmetic function $f_1(n)$ in~\eqref{eq:def:f1f2} is chosen to have $C_1(n;4N_E) = f_1(n)$.
	Replacing $U_w$ by its expression~\eqref{eq:def:Uw}, the value of the sum over $v_1$ becomes
	\[ \frac{c_7 f_1(g\ell)}{gw} X^{1/2} + O\left( \frac{X^{3/8}}{g^{1/2} w^{1/4}} \right). \]
	Recall that $\ell$ and $w$ are coprime and that $g$ divides $w$ so by multiplicativity, $f_1(g\ell) = f_1(g)f_1(\ell)$. Summing over $g$, we arrive at
	\[ V_{w,\ell}(X) = \frac{c_7 \varphi^*(w) f_1(\ell)}{w} X^{1/2} \sideset{}{'}\sum_{g \mid w} \frac{\mu(g)^2 f_1(g)}{g}
		+ O_\epsilon\left( \frac{X^{3/8}}{w^{1/4-\epsilon}} \right). \]
	The sum over $g$ satisfies
	\[ \sideset{}{'}\sum_{g \mid w} \frac{\mu(g)^2 f_1(g)}{g} = \sideset{}{'}\prod_{p \mid w} \left(1 + \frac{f_1(p)}{p} \right) = \sideset{}{'}\prod_{p \mid w} \left( 1 + \frac{p}{p^2+p-1} \right)
		= \frac{f_2(w)}{\varphi^*(w)}, \]
	which completes the proof.
\end{proof}

Write $c_8 = c_7/4N_E$. When $\gcd(\ell,w) = 1$, the equation~\eqref{eq:Nwl=Vwl} and Lemma~\ref{lem:Vwl} merge into
\begin{equation} \label{eq:Nwl-est}
	N_{w,\ell}(X) = \frac{c_8 f_2(w) \rho(\ell^2) f_1(\ell)}{w \ell^2} X^{1/2} + O(X^{3/8}).
\end{equation}
Here, the error term slightly worse than the one actually provided by the lemma. This is only to improve readability and has no consequence whatsoever, as the limitation on the range of $\alpha$ will be determined by a different error term.

\subsection{Estimating $N(\alpha;X)$}

We now have all the necessary estimates to compute a lower bound for $N_w(X)$ when $4N_E$ divides $w$, and thus also for $N(\alpha;X)$ using~\eqref{eq:N>sumNw}.
Combining~\eqref{eq:Nw>sumNwl-small} and~\eqref{eq:Nwl-est} gives, for any $\theta > 0$, the inequality
\begin{equation} \label{eq:Nw>suml}
	N_w(X) \geq \frac{c_8 f_2(w)}{w} X^{1/2} \sideset{}{'}\sum_{\substack{\ell < X^\theta \\ \gcd(\ell,w)=1}} \frac{\mu(\ell) \rho(\ell^2) f_1(\ell)}{\ell^2}
		+ O\left(\frac{X^{1/2-\theta}}{w}\right) + O(X^{3/8+\theta}) + O(X^{3/8}w^{3/4}).
\end{equation}
As previously, the prime indicates that the summation is restricted to indices coprime to $4N_E$.

Let
\begin{equation} \label{eq:def:L(w)}
	L(w) = \sideset{}{'}\sum_{\substack{\ell \geq 1 \\ \gcd(\ell,w) = 1}} \frac{\mu(\ell) \rho(\ell^2) f_1(\ell)}{\ell^2}.
\end{equation}
A simple computation shows that $L(w)$ is finite and satisfies
\begin{equation} \label{eq:suml=L(w)}
	L(w) = \sideset{}{'}\sum_{\substack{\ell \leq X^\theta \\ \gcd(\ell,w) = 1}} \frac{\mu(\ell) \rho(\ell^2) f_1(\ell)}{\ell^2} + O_\epsilon\left(\frac{1}{X^{\theta-\epsilon}}\right).
\end{equation}
We show that $L(w)$ does not vanish.

\begin{lemma} \label{lem:L(w)>c}
	There exists a constant $c_9 > 0$ satisfying $L(w) \geq c_9$ for all $w \geq 1$.
\end{lemma}

\begin{proof}
	Expanding \eqref{eq:def:L(w)} as an Euler product gives
	\[ L(w) = \sideset{}{'}\prod_{p \nmid w} \left( 1 - \frac{\rho(p^2) f_1(p)}{p^2} \right) = \sideset{}{'}\prod_{p \nmid w} \left( 1 - \frac{\rho(p^2)}{p^2+p-1} \right), \]
	with the second equality following from the definition of $f_1(p)$ in \eqref{eq:def:f1f2}.
	Recall that $\Delta_F$ denotes the discriminant of the polynomial $F(x,1)$, as defined at the beginning of Section~\ref{ssec:SFS:prelim}. We further separate this product as
	\[ L(w) = \sideset{}{'}\prod_{p \nmid \Delta_F} \left( 1 - \frac{\rho(p^2)}{p^2+p-1} \right)
		\sideset{}{'}\prod_{p \mid \Delta_F} \left( 1 - \frac{\rho(p^2)}{p^2+p-1} \right)
		\sideset{}{'}\prod_{p \mid w} \left( 1 - \frac{\rho(p^2)}{p^2+p-1} \right)^{-1}. \]
	The product over the primes dividing $w$ consists only of terms at least one, and the product over the prime divisors of $\Delta_F$ is a positive constant. Combined with \eqref{eq:rhopkrhop}, this shows
	\[ L(w) \gg \sideset{}{'}\prod_{p \nmid \Delta_F} \left( 1 - \frac{\rho(p)}{p^2+p-1} \right). \]
	Since $\rho(p) \leq 3$, the right-hand side does not vanish and the proof is complete.
\end{proof}

Gathering \eqref{eq:Nw>suml}, \eqref{eq:suml=L(w)} and Lemma~\ref{lem:L(w)>c}, we obtain
\[ N_w(X) \geq \frac{c_8c_9 f_2(w)}{w} X^{1/2} + O_\epsilon\left( \frac{f_2(w)}{w}X^{1/2-\theta+\epsilon}\right) + O\left(\frac{X^{1/2-\theta}}{w}\right)  + O(X^{3/8+\theta}) + O(X^{3/8} w^{3/4}). \]
For any choice of $\theta \in (0,1/8-4\alpha)$, we return to \eqref{eq:N>sumNw}  and sum this inequality over $w$. From the definition of $f_2(n)$ in \eqref{eq:def:f1f2}, it is immediate that $\varphi^*(w) \leq f_2(w) \ll_\epsilon w^\epsilon$, and we obtain
\[ N(\alpha;X) \geq c_8c_9 X^{1/2} \sum_{\substack{w \leq c_6X^{4\alpha} \\ w \equiv 0 \bmod 4N_E}} \frac{\varphi^*(w)}{w}
	+ O(X^{1/2}), \]
since we assume $\alpha < 1/56$.
Abel summation shows that the sum over $w$ satisfies
\[ \sum_{\substack{w \leq c_6X^{4\alpha} \\ w \equiv 0 \bmod 4N_E}} \frac{\varphi^*(w)}{w} \gg \sum_{x \ll X^{4\alpha}} \frac{\varphi^*(x)}{x} \gg \log X, \]
and we obtain the lower bound $N(\alpha;X) \gg X^{1/2} \log X$, thus completing the proof of Proposition~\ref{prop:NalphaX}.

\bibliography{bibliography}
\bibliographystyle{amsalpha}

\end{document}